\documentclass[10pt,reqno]{amsart}
\usepackage{amsmath}
\usepackage{amsthm}
\usepackage{amssymb}
\usepackage{amsfonts}
\usepackage{latexsym}
\usepackage[backref=page]{hyperref}
\usepackage{cite}

\newcommand{\B}{\mathcal{B}}

\newcommand{\D}{ \mathbb{D}}

\newcommand{\dD}{ \partial\mathbb{D}}

\newcommand{\ran}{\operatorname{ran}}

\newcommand{\Span}{\operatorname{span}}
\newcommand{\norm}[1]{\left\| #1 \right\|}
\newcommand{\inner}[1]{\langle #1 \rangle}

\newcommand{\N}{\mathbb{N}}

\newcommand{\h}{\mathcal{H}}
\newcommand{\K}{\mathcal{K}}

\renewcommand{\phi}{\varphi}

\newcommand{\WOT}{\operatorname{(WOT)}}
\newcommand{\SOT}{\operatorname{(SOT)}}
\newcommand{\SST}{\operatorname{(SST)}}

\newtheorem{Conjecture}{Conjecture}

\theoremstyle{plain}
\newtheorem{Theorem}{Theorem}

\newtheorem{Corollary}[Theorem]{Corollary}
\newtheorem*{Claim}{Claim}

\newtheorem{Lemma}[Theorem]{Lemma}
\theoremstyle{definition}
\newtheorem*{Definition}{Definition}
\newtheorem{Example}[Theorem]{Example}

\allowdisplaybreaks

\begin{document}
\bibliographystyle{plain}

    \title[On the closure of the complex symmetric operators]
    	{On the closure of the complex symmetric operators:  compact operators\\and weighted shifts}

    \author{Stephan Ramon Garcia}
    \address{   Department of Mathematics\\
            Pomona College\\
            Claremont, California\\
            91711 \\ USA}
    \email{Stephan.Garcia@pomona.edu}
    \urladdr{http://pages.pomona.edu/\textasciitilde sg064747}

	\author{Daniel E.~Poore}

    \keywords{Complex symmetric operator, unitary orbit, Hilbert space, compact operator, strong-* topology, strong operator
    	topology, weak operator topology, Kakutani shift, self-similarity, palindrome, shift operator, weight sequence, unilateral shift,
	irreducible.}
    \subjclass[2010]{47A05, 47B35, 47B99}

    \thanks{Partially supported by National Science Foundation Grant DMS-1001614.}

    \begin{abstract}
    	We study the closure $\overline{CSO}$ of the set $CSO$ of all complex symmetric operators on a
	separable, infinite-dimensional, complex Hilbert space.  Among other things, we 
	prove that every compact operator in $\overline{CSO}$ is complex symmetric. 
	Using a construction of Kakutani as motivation, we also describe many properties
	of weighted shifts in $\overline{CSO} \backslash CSO$.  In particular, we show that weighted shifts which
	demonstrate a type of approximate self-similarity belong to $\overline{CSO}\backslash CSO$.	
	As a byproduct of our treatment of weighted shifts, we explain several ways in which our result on
	compact operators is optimal.
    \end{abstract}

\maketitle

\section{Introduction}
	Throughout the following, we let $\h$ denote a separable, infinite-dimensional complex Hilbert space.
	Recall that $T \in \B(\h)$ is a \emph{complex symmetric operator} 
	if there exists a \emph{conjugation} $C$ (i.e., a conjugate-linear, isometric involution on $\h$) such that $T = CT^*C$.
	We remark that the term \emph{complex symmetric} stems from the fact that $T$ is a complex symmetric operator 
	if and only if $T$ is unitarily equivalent to a symmetric (i.e., self-transpose) matrix with complex entries,
	regarded as an operator acting on a $\ell^2$-space of the appropriate dimension 
	\cite[Sect. 2.4]{CCO}.	
	The general study of complex symmetric operators was undertaken by the first author, M.~Putinar, and W.R.~Wogen in
	\cite{CSOA, CSO2, ESCSO, MUCFO, CSPI, SNCSO}, although much of the
	theory has classical roots in the matrix-oriented work of N.~Jacobson \cite{Jacobson}, T.~Takagi \cite{Takagi}, 
	C.L.~Siegel \cite{Siegel}, and I.~Schur \cite{Schur}.  A number of other authors
	have recently made significant contributions to the study of complex symmetric operators
	\cite{CFT, Gilbreath, ZLJ, JKLL, JKLL2,WXH,Tener,Zag,ZLZ}, which has proven
	particularly relevant to the study of truncated Toeplitz operators \cite{CRW, Chalendar, NLEPHS, 
	SedlockThesis, Sedlock, TTOSIUES },
	a rapidly growing branch of function-theoretic operator theory stemming from the seminal work of D.~Sarason \cite{Sarason}.
	
	In the following, we let $CSO$ denote the set of all complex symmetric operators on $\h$.  
	We remark that the set $CSO$ is neither closed under addition nor under multiplication, 
	although it is closed under the adjoint operation and the Aluthge transform \cite[Thm.~1]{ATCSO},
	a remarkable nonlinear mapping on $\B(\h)$ which has been much studied in recent years
	\cite{Ando, Antezana, Benhida, Botelho, Cassier, Dykema, Exner, Huang, Wang}.  	
	Lately there has been some interest in the study of $CSO$ itself as a subset of $\B(\h)$
	\cite{ONCPCSO, CSPI,  ZLJ}.  Along these lines we begin by examining the closure of $CSO$
	in several of the most common topologies on $\B(\h)$ (Section \ref{SectionClosures}).
	In particular, we prove that the closure of $CSO$ in the strong-* topology is all of $\B(\h)$.  Among other things,
	this immediately implies that the strong-operator and weak-operator closures of $CSO$ are both $\B(\h)$.
	This contrasts sharply with the situation for the norm topology, 
	which we discuss in significant detail below.
	
	Let us denote by $\overline{CSO}$ the closure of $CSO$ with respect to the operator norm on $\B(\h)$.
	We first remark that $\overline{CSO}$ is a proper subset of $\B(\h)$.  Although this will be clear from what follows,
	we should mention that a simple example of an operator in $\B(\h) \backslash \overline{CSO}$ can be constructed by
	taking the direct sum of the matrix \cite[Ex.~1]{SNCSO} with $0$ and then applying \cite[Lem.~1]{CSPI}.  
	
	The so-called \emph{norm closure problem} for complex symmetric operators asked whether or not
	$\overline{CSO} = CSO$ \cite[p.~1260]{CSPI}.  Although it appeared in print only in 2009, this question had been
	circulating around the community for some years prior.		
	Recently, S.~Zhu, C.G.~Li, and Y.Q.~Ji demonstrated that a particular weighted shift operator
	belongs to $\overline{CSO} \backslash CSO$, thereby settling the norm-closure problem in the negative
	\cite{ZLJ}.  Shortly thereafter, the authors of this note constructed
	a completely different counterexample using a certain infinite direct sum of multiples of the unilateral shift and its adjoint \cite{ONCPCSO}.
	These examples indicate that the structure of $\overline{CSO}$ is much richer than previously expected.
	In particular, a complete description of the set $\overline{CSO}$ is now much desired.

\subsection*{Compact operators}
	Our first main result (Theorem \ref{TheoremCompact})
	asserts that every compact operator in $\overline{CSO}$ belongs to $CSO$.  In some sense,
	this complements the results of \cite{ONCPCSO} and \cite{ZLJ}, since none of the examples of operators in 
	$\overline{CSO} \backslash CSO$ described there are compact.  	
	Let us also remark that Theorem \ref{TheoremCompact}  furnishes a simple proof that $\overline{CSO} \neq \B(\h)$.
	Indeed, if $T$ is an irreducible weighted unilateral shift whose weights tend to zero, then $T$ is compact
	by \cite[Cor.~4.27.5]{ConwayCOT}.  However, $\dim \ker T = 0 \neq 1 = \dim \ker T^*$ for such an operator whence $T$ is
	not complex symmetric by \cite[Prop.~1]{CSOA}. 	
	
	It turns out that our result about compact operators is sharp in the following sense.
	Our proof relies heavily upon the fact that the spectrum $\sigma(|T|)$ of the \emph{modulus} $|T| = \sqrt{T^*T}$ of a compact
	operator consist of $0$ along with a decreasing sequence of positive eigenvalues of finite multiplicity.
	As we will see, it is possible to construct operators in $\overline{CSO} \backslash CSO$ such that $\sigma(|T|) \backslash \{0,1, \frac{1}{2}, \frac{1}{4},\ldots\}$ 
	consists only of eigenvalues, each of multiplicity one (Example \ref{ExampleDistinct}).
	On the other hand, the so-called \emph{Kakutani shift} (discussed at length below) belongs to $\overline{CSO} \backslash CSO$
	and satisfies $\sigma(|T|) = \{0\} \cup \{ \frac{1}{2^n} : n=0,1,2,\ldots\}$, each nonzero eigenvalue being of
	infinite multiplicity.  In light of these examples, it is difficult to envision a stronger version of Theorem \ref{TheoremCompact}.

\subsection*{Weighted shifts}
	As the preceding comments suggest, the study of $\overline{CSO}$ leads naturally
	to the consideration of weighted shifts.  Consequently, a substantial portion of this article is dedicated to this topic.
	We say that $T \in \B(\h)$ is a \emph{unilateral weighted shift}
	(or simply a \emph{weighted shift}) if there is an orthonormal basis $\{e_n\}_{n=1}^{\infty}$ of $\h$ and a sequence of scalars 
	$\{\alpha_n\}_{n=1}^{\infty}$ (the \emph{weight sequence})
	such that $Te_n = \alpha_n e_{n+1}$ for $n\geq 1$.  
	Since the weighted shift having weight sequence $\{ \alpha_n\}_{n=1}^{\infty}$ is unitarily equivalent to
	the unilateral shift with weight sequence $\{ | \alpha_n | \}_{n=1}^{\infty}$ (see \cite[Prop.~4.27.2]{ConwayCOT}
	or \cite[Prob.~89]{Halmos}), we henceforth assume that $\alpha_n \geq 0$ for all $n \geq 1$. We maintain
	this convention and the preceding notation in what follows.
	
	In light of the fact that
	\begin{equation*}
		T^*e_n =
		\begin{cases}
			0 & \text{if $n=1$},\\
			\alpha_{n-1} e_{n-1} & \text{if $n \geq 2$},
		\end{cases}
	\end{equation*}
	we see that if $\alpha_n = 0$, then $\Span\{ e_1,e_2,\ldots,e_n\}$ and $\overline{\Span\{e_{n+1},e_{n+2},\ldots\}}$
	are reducing subspaces of $T$.  Conversely, $T$ is irreducible
	if and only if $\alpha_n > 0$ for $n \geq 1$ \cite[p.~137-8]{ConwayCOT}.  For reasons which
	will become clear shortly, we focus our attention primarily on \emph{irreducible} weighted shifts.

	If the weight sequence $\alpha_n$ has exactly $N$ zeros (where $0\leq N < \infty$), then
	$\dim \ker T = N \neq N+1 = \dim \ker T^*$.  By \cite[Prop.~1]{CSOA}, this implies that $T \notin CSO$.
	We have therefore established the following lemma.

	\begin{Lemma}\label{LemmaIrreducible}
		If $T$ is an irreducible weighted shift, then $T \notin CSO$.
	\end{Lemma}	
	
	It follows that if $T$ is a complex symmetric weighted shift, then the weight $\alpha_n = 0$ occurs infinitely often.  
	In this case, $T$ is unitarily equivalent to an operator of the form $\oplus_{i=1}^{\infty} T_i$ where 
	\begin{equation}\label{eq-Palindrome}
		T_i = \small
		\begin{pmatrix}
			0 &  & & & & \\
			\alpha_1^{(i)} & 0 &  & & &\\
			& \alpha_2^{(i)} & 0 & &\\
			& &  \ddots & \ddots & & \\
			& & & \alpha_{n_i-1}^{(i)} & 0  & 
		\end{pmatrix}
	\end{equation}
	is a $n_i \times n_i$ matrix with $\alpha_j^{(i)} > 0$ for $1 \leq j \leq n_i -1$.  We can be even more precise, for
	the recent work \cite{ZL} of S.~Zhu and C.G.~Li asserts that these constants must be \emph{palindromic}, in the sense that 
	\begin{equation*}
		\alpha_j^{(i)} = \alpha_{n_i-j}^{(i)}
	\end{equation*}
	whenever $1 \leq j \leq n_i - 1$.  Conversely, any such operator $T$ is 
	complex symmetric since $T = CT^*C$ where $C = \oplus_{i=1}^{\infty} C_i$ and
	$C_i(z_1,z_2,\ldots,z_{n_i}) = ( \overline{z_{n_i}}, \overline{ z_{n_i-1} }, \ldots, \overline{ z_1 } )$.
	
	The preceding discussion suggests a method for constructing \emph{irreducible} weighted
	shifts which belong to $\overline{CSO}$.
	This was first observed by S.~Zhu, C.G.~Li, and Y.Q.~Ji in \cite{ZLJ}, who noted that the so-called
	\emph{Kakutani shift} \cite[p.~282]{Rickart} (see also \cite[Pr.~104]{Halmos}) 
	belongs to $\overline{CSO}$.  Since this is an instructive example which we shall
	frequently refer to in what follows, we recall some of the relevant details here.
	
	Let $T$ denote the weighted shift corresponding to the weight sequence $\{\alpha_n\}_{n=1}^{\infty}$ whose initial terms are given by
	\begin{equation}\label{eq-Kakutani}\small
		\underbrace{
		\underbrace{
		\underbrace{ 1,\tfrac{1}{2},1}_{2^2-1},\tfrac{1}{4},
		\underbrace{1,\tfrac{1}{2},1}_{2^2-1}}_{\text{$2^3-1$ terms}},\tfrac{1}{8},
		\underbrace{\underbrace{1,\tfrac{1}{2},1}_{2^2-1},\tfrac{1}{4},
		\underbrace{1,\tfrac{1}{2},1}_{2^2-1}}_{\text{$2^3-1$ terms}}}_{\text{$2^4-1$ terms}},\tfrac{1}{16},
		\underbrace{
		\underbrace{\underbrace{1,\tfrac{1}{2},1}_{2^2-1},
		\tfrac{1}{4},\underbrace{1,\tfrac{1}{2},1}_{2^2-1}}_{\text{$2^3-1$ terms}},
		\tfrac{1}{8},
		\underbrace{\underbrace{1,\tfrac{1}{2},1}_{2^2-1},\tfrac{1}{4},
			\underbrace{1,\tfrac{1}{2},1}_{2^2-1}}_{\text{$2^3-1$ terms}}}_{\text{$2^4-1$ terms}},\tfrac{1}{32},\ldots	.
	\end{equation}
	In particular, the sequence $\{\alpha_n\}_{n=1}^{\infty}$ can be decomposed into
	infinitely many blocks of length $2^k$, each starting with a fixed palindrome of length
	$2^k-1$ and ending with a weight $\leq \frac{1}{2^k}$.  From this perspective, it is easy to see that
	$T$ is the norm limit of complex symmetric weighted shifts.  Indeed, if $\epsilon > 0$ is given then
	the weighted shift $T'$ having weights
	\begin{equation*}
		\beta_n = 
		\begin{cases}
			\alpha_n & \text{if $\alpha_n > \epsilon$},\\
			0 & \text{if $\alpha_n \leq \epsilon$},
		\end{cases}
	\end{equation*}
	satisfies $\norm{T-T'} < \epsilon$ and is complex symmetric since it is unitarily equivalent to a direct
	sum of matrices of the form \eqref{eq-Palindrome} having palindromic weights.
	
	Let us make a few remarks about the preceding construction which will help motivate our results.
	First of all, observe that the weight sequence $\{\alpha_n\}_{n=1}^{\infty}$ of the Kakutani shift \eqref{eq-Kakutani}
	has a subsequence which converges to zero and another which tends to $\sup\{\alpha_n\}_{n=1}^{\infty}$.
	As we will see, this behavior is typical of all irreducible weighted shifts belonging to $\overline{CSO}$
	(Theorem \ref{TheoremSubsequence}).  In particular, no irreducible weighted shift in $\overline{CSO}$
	is compact.  This agrees with our earlier remarks about compact operators.
	
	It turns out that every irreducible weighted shift which is \emph{approximately Kakutaki}, in a sense
	made precise in Section \ref{SectionApproximate}, belongs to $\overline{CSO}\backslash CSO$
	(Theorem \ref{TheoremApproximate}).  Among other things, this allows us to construct
	operators in $\overline{CSO}\backslash CSO$ whose moduli have desired spectral properties.
	
\section{Closures in the weak, strong, and strong-* topologies}\label{SectionClosures}
	Although we are primarily interested in studying the closure $\overline{CSO}$ of $CSO$ with respect to 
	the operator norm, let us first say a few words about the closure of $CSO$ with respect to several 
	other standard topologies on $\B(\h)$.  In particular, we consider
	the \emph{weak operator topology} (WOT), \emph{strong operator topology} (SOT),
	and \emph{strong-* topology} (SST).  We say that\smallskip
	\begin{enumerate}\addtolength{\itemsep}{0.5\baselineskip}
		\item $T_n \to T$ means that $\lim_{n\to\infty} \norm{T - T_n} = 0$,
		\item $T_n \to T \WOT$ means that $\lim_{n\to\infty} \inner{ (T-T_n)x,y} = 0$ for all $x,y \in \h$,
		\item $T_n \to T \SOT$ means that $\lim_{n\to\infty} \norm{ (T-T_n)x} = 0$ for all $x \in \h$,
		\item $T_n \to T \SST$ means that $T_n \to T \SOT$ and $T_n^* \to T^* \SOT$.
	\end{enumerate}
	\smallskip

	Observe that the containments
	\begin{equation*}
		\overline{CSO} \subseteq
		\overline{CSO}^{\SST} \subseteq \overline{CSO}^{\SOT} \subseteq \overline{CSO}^{WOT} \subseteq \B(\h)
	\end{equation*}	
	hold trivially.  The superscripts in the preceding chain indicate the topology with respect to which the closure of $CSO$ is taken,
	the absence of a superscript being reserved for the norm topology.  Let us also remark that the preceding notions
	have obvious analogues for conjugate-linear operators, which we employ without further comment.
	
	To proceed, we require the following well-known lemma \cite[Prop.~IX.1.3.d]{ConwayCFA}:

	\begin{Lemma}\label{LemmaSOT}
		Suppose that $\{e_i\}_{i=1}^{\infty}$ is an orthonormal basis of $\h$ and that 
		$T_n$ is a bounded sequence in $B(\h)$.  If $T_n e_i\to Te_i$ for all $i$, 
		then $T_n \to T \SOT$.
	\end{Lemma}

	We are now ready to prove the main result of this section.	
	Our approach is inspired by the arguments of \cite[Ex.~2]{CSPI}, which
	were in turn inspired by a matrix trick first observed in \cite[p.~793]{VPSBF}.
	
	\begin{Theorem}\label{TheoremClosures}
		$\overline{CSO}^{\SST} = \overline{CSO}^{\SOT} = \overline{CSO}^{WOT} = \B(\h)$.
	\end{Theorem}
	
	\begin{proof}
		Since the strong and weak operator topologies are both weaker than the strong-* topology,
		it suffices to prove that $\overline{CSO}^{\SST} = \B(\h)$.  		
		Let $T \in \B(\h)$, fix an orthonormal basis $\beta=\{e_1,e_2,\ldots\}$ of $\h$, and let $\h_n = \Span\{e_1,e_2,\ldots,e_n\}$.
		Define $A_n \in \B(\h_n)$ by insisting that $\inner{A_n e_k,e_j} = \inner{T e_k,e_j}$ for $1 \leq j,k \leq n$.  In other words,
		$A_n$ is simply the upper-left $n \times n$ principal submatrix of the matrix representation of $T$
		with respect to the basis $\beta$.
		Let $C_n$ be an arbitrary conjugation on $\h_n$ and observe that the operator
		$T_n = A_n \oplus C_n A_n^* C_n \oplus 0$ on $\h$ is complex symmetric by \cite[p.~793]{VPSBF}.
		Since $n > i$ implies that
		\begin{equation}\label{eq-SimilarEstimate}
			\norm{T e_i - T_n e_i}^2 = \sum_{j=n+1}^{\infty} |\inner{T e_i,e_j}|^2 ,
		\end{equation}
		it follows that $T_n e_i \to Te_i$ for each fixed $i$.  Since $\norm{T_n} = \norm{A_n} \leq \norm{T}$ by construction,
		it follows from Lemma \ref{LemmaSOT} that $T_n \to T \SOT$.  An estimate similar to \eqref{eq-SimilarEstimate}
		and another appeal to Lemma \ref{LemmaSOT} confirm that $T_n^* \to T^* \SOT$ as well, whence
		$T_n \to T \SST$.
	\end{proof}

	Among other things, the preceding theorem implies that
	\begin{equation*}
		\overline{CSO} \subsetneq
		\overline{CSO}^{\SST} = \overline{CSO}^{\SOT} = \overline{CSO}^{WOT} = \B(\h).
	\end{equation*}	
	In particular, the closure of $CSO$ is only of interest if the closure is taken respect to the norm topology on $\B(\h)$.

\section{Compact operators}\label{SectionCompact}

	Having seen that $\overline{CSO}$ is a proper subset of $\B(\h)$, one naturally wishes to 
	know how various well-studied classes of operators intersect $\overline{CSO}$.  It turns out that a complete answer
	can be given in the case of compact operators.

	\begin{Theorem}\label{TheoremCompact}
		If $T \in \overline{CSO}$ and $T$ is compact, then $T \in CSO$.
	\end{Theorem}
	
	Before proving Theorem \ref{TheoremCompact}, we require two simple lemmas.
	
	\begin{Lemma}\label{LemmaConjugationsSOT}
		The set of all conjugations on $\h$ is SOT closed.
	\end{Lemma}
	
	\begin{proof}
		Let $C_n$ be a sequence of conjugations on $\h$ such that $C_n \to C \SOT$.
		For each $x \in \h$ we have
		\begin{align*}
			\norm{C^2 x  - x}
			&= \norm{ C^2 x - C_n^2 x } \\
			&\leq \norm{ C^2 x - C_n C x} + \norm{ C_n C x - C_n^2 x} \\
			&= \norm{ (C -  C_n) C x} + \norm{ (C - C_n) x},
		\end{align*}
		which tends to zero by hypothesis whence $C^2 = I$.
		Next observe that
		\begin{equation*}
			| \norm{Cx} - \norm{x} |
			= | \norm{Cx} - \norm{C_n x} | 
			\leq \norm{Cx - C_n x} \to 0,
		\end{equation*}
		from which it follows that $C$ is isometric.  Since $C$ is obviously conjugate-linear, we conclude that 
		$C$ is a conjugation on $\h$.
	\end{proof}	

	\begin{Lemma}\label{LemmaConjugations}
		If $T \in \overline{CSO}$, then there exist conjugations $C_n$ such that $C_nT^*C_n\!\to\!T$.
	\end{Lemma}
	
	\begin{proof}
		If $T_n$ is a sequence of operators such that $T_n \to T$ and $C_n$ is a sequence of
		conjugations such that $T_n = C_n T_n^* C_n$, then
		\begin{align*}
			\norm{T - C_n T^* C_n}
			&\leq \norm{T - C_n T_n^*C_n} + \norm{C_n T_n^* C_n - C_n T^* C_n} \\
			&= \norm{T - T_n} + \norm{T_n^* - T^*} \\
			&= 2\norm{T - T_n}
		\end{align*}
		tends to zero.
	\end{proof}

	\begin{proof}[Pf.~of Theorem \ref{TheoremCompact}]
		Suppose that $T$ belongs to $\overline{CSO}$.  
		Let $\K = \overline{ \ran T + \ran T^*}$ so that $\K^{\perp} \subseteq \ker T \cap \ker T^*$ is a reducing
		subspace upon which $T$ vanishes.  In other words, with respect to the orthogonal
		decomposition $\h = \K \oplus \K^{\perp}$ we have $T = T|_{\K} \oplus 0$.  By \cite[Lem.~1]{CSPI}
		we know that $T \in \B(\h)$ is complex symmetric if and only if $T|_{\K} \in \B(\K)$
		is complex symmetric.  Without loss of generality, we therefore assume that
		\begin{equation}\label{eq-Density}
			\overline{ \ran T + \ran T^* } = \h.
		\end{equation}
		
		Let us briefly discuss our main approach.
		Since $T$ belongs to $\overline{CSO}$, there exist conjugations $C_n$ such that $C_n T^* C_n \to T$ 
		by Lemma \ref{LemmaConjugations}.  
		It suffices to prove that there exists a subsequence $C_{n_j}$ of the $C_n$ which converges
		pointwise on both $\ran T$ and $\ran T^*$.  Indeed, the uniform boundedness of the $C_{n_j}$
		and the assumption \eqref{eq-Density} will then ensure that $C_{n_j}$ is SOT convergent on all of $\h$, whence
		there exists a conjugation $C$ such that $C_{n_j} \to C \SOT$
		by Lemma \ref{LemmaConjugationsSOT}.  The desired conclusion $T = CT^*C$ will then follow 
		by a simple limiting argument.
		
		By \cite[VIII.3.11]{ConwayCFA}, there exists a partial isometry
		$U$ with $\ker T = \ker U = \ker |T|$ such that $T = U|T|$.
		Here $|T|$ denotes the compact selfadjoint operator $\sqrt{T^*T}$.
		Let $\lambda_1 \geq \lambda_2 \geq \cdots > 0$ denote the nonzero eigenvalues of $|T|$,
		repeated according to their multiplicity, and let $x_1,x_2,\ldots$ denote corresponding
		orthonormal eigenvectors.  The vectors $y_i = Ux_i$ for $i = 1,2,\ldots$ are also orthonormal and we see that 
		\begin{equation}\label{eq-Ranges}
			\overline{\Span\{y_1,y_2,\ldots \}} = \overline{\ran T}, 
			\qquad	\overline{\Span\{x_1,x_2,\ldots \}} = \overline{\ran T^*}.
		\end{equation}
		By the remarks of the preceding paragraph, it suffices to find a subsequence $C_{n_j}$ of $C_n$
		such that $\lim_{j\to\infty} C_{n_j} x_m$ and $\lim_{j\to\infty} C_{n_j} y_m$ exist for
		$m=1,2,\ldots$. 
	
		By compactness, the eigenspaces for $|T|$ corresponding to its nonzero
		eigenvalues are finite-dimensional and hence we may define a sequence 
		$\ell_k$ of indices such that $\lambda_{\ell_k} > \lambda_{\ell_{k+1}}$ 
		and $\lambda_{\ell_k} = \lambda_i$ for $\ell_k \leq i < \ell_{k+1}$.
		Let $P_k$ denote the orthogonal projection onto the spectral subspace
		\begin{equation}\label{eq-PkSpace}
			\ker(|T| - \lambda_{\ell_k}I) = \Span\{x_{\ell_k}, x_{\ell_k + 1}, ..., x_{\ell_{k+1} - 1} \}  
		\end{equation}
		and observe that $Q_k = UP_kU^*$ is the orthogonal projection onto
		\begin{equation*}	
			\Span\{y_{\ell_k},y_{\ell_k + 1}, ..., y_{\ell_{k+1} - 1} \}.
		\end{equation*}
		The most difficult step in the proof of Theorem \ref{TheoremCompact} is the verification
		of the following claim:
	
		\begin{Claim}
			For $\ell_k \leq m < \ell_{k+1}$, we have 
			\begin{align}
				\lim_{n \to \infty}\norm{P_k C_n y_m} &=1,\label{eq-MegaClaim01}\\
				\lim_{n \to \infty}\norm{Q_k C_n x_m} &= 1.\label{eq-MegaClaim02}
			\end{align}
		\end{Claim}
		
		\begin{proof}[Pf.~of Claim]
			We proceed by induction on $k$.  Suppose for our inductive hypothesis that 
			\begin{equation}\label{eq-PCUxm}
				\lim_{n \to \infty}\norm{P_j C_n y_i} = \lim_{n \to \infty}\norm{Q_j C_n x_i}  =1
			\end{equation}
			holds for all $\ell_j \leq i < \ell_{j+1}$ whenever $0 < j < k$.  Observe that this is trivially true
			if $k = 1$ since no corresponding $j$ exist.  
			
			Let $\ell_k \leq m < \ell_{k+1}$ and for each fixed $i < \ell_k$
			let $j$ be the unique index $j < k$ such that $\ell_j \leq i < \ell_{j+1}$.
			By the inductive hypothesis \eqref{eq-PCUxm}, we see that 
			\begin{equation}\label{eq-IPCy}
				\lim_{n \to \infty}(I - P_j) C_n y_i = 0
			\end{equation}
			since $C_ny_i$ is a unit vector.  For $i < \ell_k$ it follows from \eqref{eq-IPCy} that
			\begin{align}
				\lim_{n \to \infty} \inner{C_n x_m, y_i} 
				&= \lim_{n \to \infty} \inner{C_n y_i ,x_m } \nonumber \\
				&= \lim_{n \to \infty}\inner{P_j C_n y_i,  x_m} + \lim_{n \to \infty} \inner{(I - P_j) C_n y_i,  x_m} \nonumber \\
				&= \lim_{n \to \infty} \inner{C_n y_i,  0} +\inner{0,  x_m} \nonumber \\
				&= 0 \label{eq-Cnxmyi}
			 \end{align}	 
			 since $P_j x_m = 0$ by definition.
		
			Suppose that $0 < \epsilon < 4 \lambda_{\ell_k}^2$.
			Since $C_nT^*C_n \to T$ there exists some $N_1$ so that 
			\begin{equation*}
				n \geq N_1 \quad \Rightarrow \quad \norm{C_n T^*C_n x_m  - Tx_m} 
				< \frac{\epsilon}{4\lambda_{\ell_k}}
			\end{equation*}
			holds whenever $\ell_k \leq m < \ell_{k+1}$. 
			On the other hand, by \eqref{eq-Cnxmyi} there exists some $N_2$ such that 
			\begin{equation}\label{eq-SeriesE2}
				n \geq N_2 \quad \Rightarrow \quad
				\sum_{i=1}^{\ell_k - 1} \lambda_i^2 |\inner{C_n x_m, y_i}|^2 < \frac{\epsilon}{2}.
			\end{equation}
			Since $\norm{Tx_m} = \lambda_{\ell_k}$, for $n \geq N = \max\{N_1, N_2\}$ it follows that
			\begin{align*}
				\lambda_{\ell_k}^2 - \frac{\epsilon}{2}
				&< \lambda_{\ell_k}^2 - \frac{\epsilon}{2} + \frac{\epsilon^2}{16\lambda_{\ell_k}^2} \\
				&< \left(\lambda_{\ell_k} - \frac{\epsilon}{4\lambda_{\ell_k}} \right)^2 \\
				&\leq \left( \norm{ Tx_m} - \norm{C_n T^* C_n x_m - Tx_m} \right)^2 \\
				&\leq \norm{C_nT^*C_nx_m}^2 \\
				&=\norm{ T^*C_n x_m}^2 \\
				& = \sum_{i=1}^\infty  |\inner{T^*C_nx_m, x_i}|^2 && \text{by \eqref{eq-Ranges}}\\
				& = \sum_{i=1}^\infty   |\inner{|T|U^*C_nx_m, x_i}|^2 \\
				& = \sum_{i=1}^\infty   \lambda_i^2 |\inner{U^*C_nx_m, x_i}|^2 \\
				& = \sum_{i=1}^\infty \lambda_i^2  |\inner{C_nx_m, y_i}|^2 \\
				&< \frac{\epsilon}{2} + \sum_{i = \ell_k}^\infty \lambda_i^2 |\inner{C_n x_m, y _i}|^2 && \text{by \eqref{eq-SeriesE2}} \\
				&< \frac{\epsilon}{2}+ \lambda_{\ell_k}^2 \underbrace{ \sum_{i = \ell_k}^{\ell_{k+1}-1}  |\inner{C_n x_m, y _i}|^2 }_{d_{n,m}}+ 
					\lambda_{\ell_{k+1}}^2  \sum_{i = \ell_{k+1}}^\infty  |\inner{C_n x_m, y _i}|^2  \\
				&\leq \frac{\epsilon}{2}+ \lambda_{\ell_k}^2 d_{n,m}+
					\lambda_{\ell_{k+1}}^2 (1 - d_{n,m}).
			\end{align*}
			The final inequality follows since $C_n x_m$ is a unit vector and $y_{\ell_k}, y_{\ell_k+1},\ldots$ is an orthonormal set.
			Rearranging things somewhat we see that
			\begin{equation}\label{eq-DecreasingToZero}
				n\geq N \quad \Rightarrow \quad 0 \leq (\underbrace{\lambda_{\ell_k}^2 - \lambda_{\ell_{k+1}}^2}_{>0}) (1-d_{n,m}) < \epsilon,
			\end{equation}
			whence 
			\begin{equation}\label{eq-dnm}
				\lim_{n\to\infty} d_{n,m} = 1
			\end{equation}
			whenever $\ell_k \leq m < \ell_{k+1}$.  Since
			\begin{equation*}
				d_{n,m} = \sum_{i = \ell_k}^{\ell_{k+1}-1}  |\inner{ C_n x_m, Q_k y _i}|^2 = \norm{Q_k C_n x_m}^2,
			\end{equation*}
			we obtain the second condition \eqref{eq-MegaClaim02} of the claim.  To conclude the induction,
			we need to establish the first condition \eqref{eq-MegaClaim01}.
		
			Summing \eqref{eq-dnm} over $\ell_k \leq i < \ell_{k+1}$ we obtain
			\begin{align*}
				\ell_{k+1} - \ell_k
				&= \lim_{n\to\infty}\sum_{i = \ell_k}^{\ell_{k+1} - 1} d_{n,i} \\
				&= \lim_{n\to\infty} \sum_{i = \ell_k}^{\ell_{k+1} - 1} \sum_{m = \ell_k}^{\ell_{k+1} - 1} |\inner{C_nx_i, y_m}|^2 \ \\
				&=  \sum_{m = \ell_k}^{\ell_{k+1} - 1} \left(\lim_{n\to\infty}\sum_{i = \ell_k}^{\ell_{k+1} - 1}|\inner{C_ny_m, x_i}|^2  \right)\\
				&\leq \sum_{m = \ell_k}^{\ell_{k+1} - 1} 1  \\
				&= \ell_{k+1} - \ell_k
			\end{align*}
			by applying Bessel's inequality to the unit vector $C_n y_m$ and using the fact that
			set $x_{\ell_k},x_{\ell_k+1},\ldots, x_{\ell_{k+1}-1}$ is orthonormal.  In particular, the preceding tells us that
			\begin{equation*}
				\lim_{n\to\infty} \norm{P_k C_n y_m}^2
				=\lim_{n\to\infty}\sum_{i = \ell_k}^{\ell_{k+1} - 1}|\inner{C_ny_m, P_k x_i}|^2 = 1,
			\end{equation*}
			which is the first condition \eqref{eq-MegaClaim01} of the claim.  This concludes the proof of the claim.
		\end{proof}
		
		We now wish to prove that there exists a subsequence $C_{n_j}$ of $C_n$ such that
		$C_{n_j} x_m$ and $C_{n_j} y_m$ converge for each $m=1,2,3,\ldots$.
		For each fixed $m$ there exists a unique $k$ such that $\ell_k \leq m < \ell_{k+1}$.
		Since the sets $\{ P_k C_n y_m : n = 1,2,\ldots\}$ and $\{  Q_k C_n x_m: n=1,2,\ldots\}$
		are bounded subsets of the finite-dimensional spaces $\ran P_k$ and $\ran Q_k$, respectively,
		it follows from a standard subsequence refinement argument that there exists vectors $y'_m \in \ran P_k$ and $x'_m \in \ran Q_k$ and 
		a subsequence $C_{n_j}$ of $C_n$ such that 
		\begin{align}
			\lim_{j\to\infty} P_k C_{n_j} y_m &= y'_m,\label{eq-PkCnjymCym}\\
			\lim_{j\to\infty} Q_k C_{n_j} x_m &= x'_m.\label{eq-QkCnjxmCxm}
		\end{align}
		Putting this all together we find that
		\begin{align*}
			\lim_{j\to\infty} \norm{C_{n_j} y_m - y'_m}
			&\leq\lim_{j\to\infty}  \norm{(I -P_k)C_{n_j} y_m } + \lim_{j\to\infty}\norm{ P_k C_{n_j} y_m - y'_m} \\
			&=\lim_{j\to\infty}  \sqrt{ \norm{C_{n_j} y_m}^2 - \norm{ P_k C_{n_j} y_m}^2} + 0\\
			&= \sqrt{ 1 - \lim_{j\to\infty} \norm{ P_k C_{n_j} y_m}^2} \\
			&=0
		\end{align*}
		by \eqref{eq-PkCnjymCym} and \eqref{eq-MegaClaim01}, respectively.  Thus $\lim_{j\to\infty} C_{n_j} y_m = y'_m$,
		as desired.  An analogous argument confirms that $\lim_{j\to\infty} C_{n_j} x_m = x'_m$ as well.  This concludes the proof
		of Theorem \ref{TheoremCompact}.
	\end{proof}

\section{Weighted shifts}\label{SectionShifts}

	We turn our attention now toward weighted shifts.  	
	It turns out that many features of the Kakutani shift \eqref{eq-Kakutani} are typical of irreducible weighted shifts which belong to
	$\overline{CSO}$.  For instance, consider the following theorem.
	
	\begin{Theorem}\label{TheoremSubsequence}
		If $T \in \overline{CSO}$ is an irreducible weighted shift with weights $\{ \alpha_n \}_{n=1}^{\infty}$, then
		\begin{enumerate}\addtolength{\itemsep}{0.5\baselineskip}
			\item there exists a subsequence of $\{ \alpha_n \}_{n=1}^{\infty}$ which tends to zero,
			\item there exists a subsequence of $\{ \alpha_n \}_{n=1}^{\infty}$ which tends to $\alpha_+ = \sup \{\alpha_n\}_{n=1}^{\infty}$.
		\end{enumerate}
		\smallskip
		In particular, $0$ and $\alpha_+$ belong to the essential spectrum $\sigma_{\text{e}}(|T|)$ of $|T| = \sqrt{T^*T}$. 
	\end{Theorem}
	
	Recalling that a weighted shift $T$ with weights $\{ \alpha_n \}_{n=1}^{\infty}$
	is compact if and only if $\alpha_n \to 0$ \cite[Cor.~4.27.5]{ConwayCOT}, we see that Theorem 
	\ref{TheoremSubsequence} asserts that there are no compact irreducible weighted shifts in 
	$\overline{CSO}$, in agreement with Theorem \ref{TheoremCompact}.

	The remainder of this section is devoted to developing the tools required to prove Theorem \ref{TheoremSubsequence}.
	In particular, we prove statements (i) and (ii) separately since they call for completely different methods.
	The final statement of Theorem \ref{TheoremSubsequence}, however, can easily be justified since statements (i)
	and (ii) imply that neither $0$ nor $\alpha_+$ is an isolated eigenvalue of $|T|$ of finite
	multiplicity \cite[Prop.~4.6]{ConwayCFA}.

	Before proving the first portion of Theorem \ref{TheoremSubsequence}, we need to introduce a few useful facts 
	about the spectral theory of operators in $CSO$ and its closure.
	Recall that a complex number $\lambda$ belongs to the \emph{approximate point spectrum}
	$\sigma_{\text{ap}}(T)$ of $T$ if and only if $T - \lambda I$ is not bounded below.  In other words, 
	$\lambda$ belongs to $\sigma_{\text{ap}}(T)$ if and only if there exists a sequence of unit vectors
	$x_n$ such that $(T - \lambda I)x_n \to 0$.  This is equivalent
	to asserting that $T$ is not left invertible in $\B(\h)$ \cite[p.~116]{ConwayCOT} 
	(see also \cite[Prop.~VII.6.4, Ex.~VII.3.4]{ConwayCFA}).  Although in general, one only has 
	\begin{equation}\label{eq-HalmosSpectrum}
		\sigma_{\text{ap}}(T) \cup \overline{\sigma_{\text{ap}}(T^*) } = \sigma(T),
	\end{equation}
	(see \cite[Pr.~73]{Halmos}) for a complex symmetric operator one obtains something significantly stronger \cite[Lem.~4.1]{JKLL}.  
	Indeed, if $T = CT^*C$ for some conjugation $C$, then observe that 
	$\norm{(T - \lambda I)x} = \norm{ (T^* - \overline{\lambda} I)Cx}$ for all $x \in \h$ whence
	\begin{equation*}
		\sigma_{\text{ap}}(T) = \overline{ \sigma_{\text{ap}}(T^*)}.
	\end{equation*}
	Putting the preceding together with \eqref{eq-HalmosSpectrum} we find that
	\begin{equation}\label{eq-SpectrumMagic}
		\sigma(T) = \sigma_{\text{ap}}(T).
	\end{equation}
	It turns out that \eqref{eq-SpectrumMagic} also holds under the weaker assumption
	that $T$ is a norm limit of complex symmetric operators.

	\begin{Theorem}\label{TheoremApproximatePointSpectrum}
		If $T \in \overline{CSO}$, then $\sigma(T) = \sigma_{\text{ap}}(T) = \overline{ \sigma_{\text{ap}}(T^*)}$.
	\end{Theorem}
	
	\begin{proof}
		By Lemma \ref{LemmaConjugations} there exists conjugations $C_n$ such that $C_n T^* C_n \to T$.
		If $\lambda$ belongs to $\sigma_{\text{ap}}(T)$, then there exists a sequence $x_n$ of unit vectors such that
		$(T - \lambda I)x_n \to 0$.  Thus
		\begin{align*}
			\norm{(T^*-\overline{\lambda}I) C_n x_n }
			&= \norm{ C_n T^* C_n x_n - \lambda  x_n} \\
			&\leq \norm{ (C_n T^* C_n  -T)x_n} + \norm{(T-\lambda I)x_n},
		\end{align*}
		which tends to zero.  We therefore conclude that  
		$\sigma_{\operatorname{ap}}(T) \subseteq \overline{\sigma_{\operatorname{ap}}(T^*)}$.
		The proof of the reverse containment is similar.
	\end{proof}

	The preceding theorem gives us a simple criterion for  excluding certain operators from $\overline{CSO}$.
	For instance, the unilateral shift $S$ is not a norm limit of complex symmetric operators 
	since $\sigma(S) = \overline{\D}$ but $\sigma_{\text{ap}}(S) = \dD$ \cite[Prob.~82]{Halmos}
	(of course there are more direct ways to prove this, see \cite[Ex.~2.14]{CCO}, or \cite[Cor.~7]{MUCFO}).

	With these preliminaries in hand, the proof of part (i) of Theorem \ref{TheoremSubsequence} is now quite simple.

	\begin{proof}[Pf.~of Theorem \ref{TheoremSubsequence}, (i)]
		Since $T^* e_1 = 0$ it follows that $0 \in \sigma_{\text{ap}}(T^*)$ whence $0 \in \sigma_{\text{ap}}(T)$
		by Theorem \ref{TheoremApproximatePointSpectrum}.  Thus there exists a sequence of unit vectors $x_n$
		such that $Tx_n \to 0$.  Suppose toward a contradiction that there exists some $\delta > 0$ such that
		$\alpha_n > \delta$ for all $n$.  This implies that
		\begin{equation*}
			\norm{Tx_n}^2
			= \sum_{i=1}^{\infty} \alpha_i^2 |\inner{x_n,e_i}|^2 
			\geq \delta^2 \sum_{i=1}^{\infty} |\inner{x_n,e_i}|^2 
			= \delta^2\norm{x_n}^2 = \delta^2,
		\end{equation*}
		which contradicts the fact that $Tx_n \to 0$.  
	\end{proof}

	Before proving the second part of Theorem \ref{TheoremSubsequence}, we
	require a somewhat lengthy technical lemma and the following definition.

	\begin{Definition}
		If $T\in \B(\h)$ and $x \in \h$, then we say that $x$ is \emph{shift-cyclic} for $T$ if
		\begin{equation*}
			\overline{ \Span\{x, Tx, T^2x,\ldots, T^*x, {T^*}^2x,\ldots\} } = \h.
		\end{equation*}
	\end{Definition}
	
	The motivation for the preceding definition lies in the fact if $T$ is an irreducible weighted shift, then
	each corresponding basis vector $e_n$ of $\h$ is shift-cyclic for $T$. 
	
	\begin{Lemma}\label{LemmaShiftTechnical}
		Suppose that $T \in \overline{CSO}$.  If $C_n$ is a sequence of conjugations such that
		$C_n T^* C_n \to T$, $x$ is a shift-cyclic vector for $T$, and $C_n x$ converges, then $T \in CSO$.
	\end{Lemma}
	
	\begin{proof}
		Let $v \in \h$ and $\epsilon > 0$.
		Without loss of generality we may assume that $\norm{T} \leq 1$, $\norm{x}=1$, and $\norm{v} = 1$.
		Since $x$ is a shift-cyclic vector for $T$, 
		there exists constants $a_0,a_1,\ldots,a_m$ and $b_0,b_1,\ldots,b_m$ such that
		\begin{equation*}
			\norm{v - \left(\sum_{k=0}^m a_k T^k x + b_k {T^*}^k x\right)} < \frac{\epsilon}{6}.
		\end{equation*}
		Since each $C_n$ is a conjugation it follows that
		\begin{equation}\label{eq-SixPart01}
			\norm{C_n v - \left(\sum_{k=0}^m \overline{a_k} C_nT^k   + \overline{b_k}C_n{T^*}^k  \right)x}  < \frac{\epsilon}{6}.
		\end{equation}
		Since $\Delta_n = T - C_nT^*C_n \to 0$ by hypothesis, there exists $N_1$ such that
		\begin{equation*}
			n \geq N_1 \quad\Rightarrow \quad \norm{ \Delta_n} < \min\left\{ 1 , \frac{\epsilon}{6M2^{m}} \right\},
		\end{equation*}
		where
		\begin{equation*}
			M = \sum_{k=0}^m (|a_k|+|b_k|).  
		\end{equation*}
		In particular, $n \geq N_1$ implies that
		\begin{align*}
			\norm{C_n T^k - {T^*}^kC_n} 
			&= \norm{T^k - C_n{T^*}^kC_n}  \\
			&= \norm{(C_nT^*C_n + \Delta_n)^k - (C_nT^*C_n)^k}  \\
			&\leq \sum_{j=1}^k \binom{k}{j}\norm{C_nT^*C_n}^{k-j} \norm{\Delta_n}^{j}  \\
			&< \norm{\Delta_n} \sum_{j=1}^k \binom{k}{j}\norm{T}^{k-j} \norm{\Delta_n}^{j-1} \\
			&< \norm{ \Delta_n} 2^k  \\
			&< \frac{ \epsilon}{6M2^{m-k}}\\
			&< \frac{\epsilon}{6M}
		\end{align*}
		holds for $0 \leq k \leq m$.  Thus for $n \geq N_1$ we have
		\begin{align}
			&\norm{\sum_{k=0}^m (\overline{a_k}C_nT^k  + \overline{b_k}C_n{T^*}^k  )x -
			\sum_{k=0}^m (\overline{a_k}{T^*}^k C_n  + \overline{b_k}T^kC_n )x} \nonumber \\
			&\qquad\leq \norm{\sum_{k=0}^m \overline{a_k}(C_nT^k - {T^*}^kC_n)x} 
				+\norm{\sum_{k=0}^m \overline{b_k}(C_n{T^*}^k - T^kC_n)x}\nonumber \\
			&\qquad\leq\sum_{k=0}^m (|a_k| + |b_k|)\norm{C_n T^k - {T^*}^kC_n} \nonumber \\
			&\qquad < M\cdot \frac{\epsilon}{6M}   \nonumber \\
			&\qquad = \frac{\epsilon}{6} . \label{eq-SixPart02}
		\end{align}
		Since $C_n x$ converges to some $y$ by assumption, there exists $N_2$ such that
		\begin{equation*}
			n \geq N_2 \quad \Rightarrow \quad \norm{C_n x - y} < \frac{ \epsilon}{6M}.
		\end{equation*}
		Therefore $n \geq N_2$ implies that
		\begin{align}
			&\norm{\sum_{k=0}^m (\overline{a_k}{T^*}^k C_n+ \overline{b_k}T^k C_n) x 
				- \sum_{k=0}^m (\overline{a_k}{T^*}^k + \overline{b_k}T^k)y}\nonumber \\
			&\qquad\leq \sum_{k=0}^m \norm{\overline{a_k}{T^*}^k + \overline{b_k}T^k} \norm{C_n x - y} \nonumber\\
			&\qquad <  M\cdot \frac{\epsilon}{6M}    \nonumber\\
			&\qquad =  \frac{\epsilon}{6}. \label{eq-SixPart03}
		\end{align}
		Putting this all together, if $n \geq N=\max\{N_1,N_2\}$ we find that
		\begin{align*}
			\norm{C_n v -\left(\sum_{k=0}^m \overline{a_k} {T^*}^k  + \overline{b_k}T^k \right)y } <  \frac{\epsilon}{2}
		\end{align*}
		by \eqref{eq-SixPart01}, \eqref{eq-SixPart02}, and \eqref{eq-SixPart03}.  In particular,
		\begin{equation*}
			n,n' \geq N \quad \Rightarrow \quad \norm{C_n v - C_{n'}v } < \epsilon
		\end{equation*}
		whence $C_n v$ is Cauchy and therefore converges.  By Lemma \ref{LemmaConjugationsSOT}, 
		there exists a conjugation $C$ such that $C_n \to C \SOT$.
	
		Now consider $T - CT^*C$.  Fix $u \in \h$ and observe that
		\begin{align*}
			\norm{(T - CT^*C)u}
			&=\norm{(CT - T^*C)u} \\
			&\leq \norm{(C - C_n)Tu} + \norm{C_n T u - T^* C_n u} + \norm{T^* (C_n-C) u} \\
			&\leq \norm{(C - C_n)Tu} + \norm{T - C_n T^* C_n} + \norm{(C_n-C) u}.
		\end{align*}
		Since $C_n \to C \SOT$, the first and third terms tend to zero.  The second term tends to zero
		by hypothesis whence $T = CT^*C$ so that $T \in CSO$.
	\end{proof}

	Now armed with Lemma \ref{LemmaShiftTechnical}, we complete the proof of Theorem \ref{TheoremSubsequence}.

	\begin{proof}[Pf.~of Theorem \ref{TheoremSubsequence}, (ii)]
		Suppose toward a contradiction that no such subsequence exists.  Letting $\alpha_+ = \sup \{\alpha_n\}_{n=1}^{\infty}$,
		it follows that there exists an index $\ell$ such that $\alpha_{\ell} = \alpha_+$.  Additionally, 
		there exist $\delta > 0$ and $N \in \N$ such that
		\begin{equation}\label{eq-Mdelta}
			n \geq N \quad \Rightarrow 0 < \alpha_n \leq \alpha_+ - \delta.
		\end{equation}
		Since $T \in \overline{CSO}$, there exist conjugations $C_n$ such that $C_n T^* C_n \to T$
		by Lemma \ref{LemmaConjugations}.  We now write 
		\begin{equation*}
			C_n e_{\ell} = \underbrace{\sum_{k=1}^N \inner{C_n e_{\ell},e_k}e_k}_{x_n} 
			+ \underbrace{ \sum_{k=N+1}^{\infty} \inner{C_n e_{\ell},e_k} e_k }_{y_n}
		\end{equation*}
		and observe that $\norm{x_n}^2+ \norm{y_n}^2= 1$ whence
		\begin{align*}
			\norm{C_nT^*C_ne_\ell}^2 
			&= \norm{T^*(x_n+y_n)}^2 \\
			&\leq \alpha_+^2 \norm{x_n}^2  + (\alpha_+ - \delta)^2 \norm{y_n}^2\\
			&\leq \alpha_+^2 \norm{x_n}^2  + (\alpha_+^2 -  2\alpha_+ \delta + \delta^2) \norm{y_n}^2\\
			&= \alpha_+^2  - \delta (2\alpha_+ -\delta) \norm{y_n}^2\\
			&\leq \alpha_+^2   - \delta \alpha_+ \norm{y_n}^2
		\end{align*}
		since $\delta < \alpha_+$ by \eqref{eq-Mdelta}.
		In particular, the preceding tells us that $\norm{C_nT^*C_ne_\ell} \leq \alpha_+$ from which it follows that
		\begin{align*}
			0
			&\leq \delta \alpha_+ \norm{y_n}^2\\
			&\leq \alpha_+^2 - \norm{C_nT^*C_ne_\ell}^2 \\
			&= \norm{T e_{\ell}}^2 - \norm{C_nT^*C_ne_\ell}^2 \\
			&= (\norm{T e_{\ell}} + \norm{C_nT^*C_ne_\ell})(\norm{T e_{\ell}} - \norm{C_nT^*C_ne_\ell}) \\
			&\leq 2\alpha_+(\norm{T e_{\ell}} - \norm{C_nT^*C_ne_\ell}) \\
			&\leq 2\alpha_+\norm{(T  - C_nT^*C_n)e_\ell}.
		\end{align*}
		Since the preceding tends to zero, we conclude that $y_n \to 0$.
	
		Now observe that the vectors $x_n$ belong to unit ball of the finite-dimensional
		space $\h_N = \Span\{e_i\}_{i=1}^N$.  Thus there exists a subsequence $x_{n_k}$ of the $x_n$
		which converges to some $x \in \h_N$.  Therefore
		\begin{equation*}
			\norm{C_{n_k}e_{\ell} - x} = \norm{x_{n_k} + y_{n_k} - x} \leq \norm{x_{n_k} - x} + \norm{y_{n_k}} \to 0,
		\end{equation*}
		whence $C_{n_k}e_{\ell} \to x$.  Since $e_{\ell}$ is a shift-cyclic vector for $T$ and 
		$C_{n_k}T^*C_{n_k} \to T$, we conclude from Lemma \ref{LemmaShiftTechnical} that $T\in CSO$.
		However, this contradicts Lemma \ref{LemmaIrreducible}.
	\end{proof}

\section{Approximately Kakutani shifts}\label{SectionApproximate}

	As we saw in Section \ref{SectionShifts},
	the Kakutani shift \eqref{eq-Kakutani} demonstrates behavior which is  typical of irreducible
	weighted shifts in $\overline{CSO}$.  While Theorem \ref{TheoremSubsequence} addresses some of the large-scale
	structure of the weight sequence $\{ \alpha_n \}_{n=1}^{\infty}$, it sheds little light on the small-scale behavior of the weights.  
	For instance, the Kakutani shift possesses a remarkable self-similarity in the sense that certain palindromic sequences
	are repeated infinitely often in its weight sequence.  Remarkably, it turns out that 
	an irreducible weighted shift which demonstrates some approximate level of self-similarity
	must belong to $\overline{CSO}$.
	
	\begin{Theorem}\label{TheoremApproximate}
		If $T$ is an irreducible weighted shift with weights $\{ \alpha_n\}_{n=1}^{\infty}$ such that 
		for each $n \in \N$ and $\epsilon > 0$ there exists an index $c_{n,\epsilon}\geq n$ such that
		\begin{equation}\label{eq-Trapped}
			0<\alpha_{c_{n,\epsilon}} < \epsilon,
		\end{equation}
		and 
		\begin{equation}\label{eq-AddingIndices}
			1 \leq k \leq n \quad \Rightarrow \quad |\alpha_k - \alpha_{c_{n,\epsilon}-k}| < \epsilon,
		\end{equation}
		then $T \in \overline{CSO}$.
	\end{Theorem}
	
	Since the proof of the preceding theorem is somewhat long and involved, we defer it until the
	end of this section. We instead prefer to focus on a related conjecture and several consequences of our theorem.
	
	Let us call an irreducible weighted shift $T$ satisfying the hypotheses of Theorem \ref{TheoremApproximate}
	\emph{approximately Kakutani}.  We conjecture that this property is also necessary for an irreducible
	weighted shift to belong to $\overline{CSO}$.
	
	\begin{Conjecture}
		Every irreducible weighted shift in $\overline{CSO}$
		is approximately Kakutani.
	\end{Conjecture}
	
	Among other things, the following corollary asserts that an irreducible weighted shift whose
	weight sequence is a suitable perturbation of the Kakutani sequence \eqref{eq-Kakutani} 
	also belongs to $\overline{CSO}$.  In particular, this permits us to construct weighted shifts
	in $\overline{CSO} \backslash CSO$ whose moduli have desired spectral properties.

	\begin{Corollary}\label{CorollaryKakutani}
		If $T$ is an irreducible weighted shift with weights $\{ \alpha_n\}_{n=1}^{\infty}$ such that 
		\begin{enumerate}\addtolength{\itemsep}{0.5\baselineskip}
			\item $\lim_{n\to\infty} \alpha_{2^n} = 0$,
			\item $\lim_{n\to\infty} \sup\{ | \alpha_k - \alpha_{2^n-k}|: 1 \leq k \leq 2^n\} = 0$, 
		\end{enumerate}
		\smallskip
		then $T$ belongs to $\overline{CSO}$.
	\end{Corollary}

	\begin{proof}
		Fix $n$ and let $\epsilon > 0$.  By (i), there exists $K_1$ such that 
		$0 < \alpha_{2^k} < \epsilon$ holds whenever $k \geq K_1$.  
		Letting 
		\begin{equation*}
			A_n= \sup\{ | \alpha_k - \alpha_{2^n-k}|: 1 \leq k \leq 2^n\},
		\end{equation*}
		we obtain from (ii) that there is a $K_2$ such that $k \geq K_2$ implies that 
		$0 \leq A_k < \epsilon$.  Now let 
		\begin{equation*}
			c_{n,\epsilon} = 2^{K_1+K_2+n}.
		\end{equation*}
		Since $K_1+K_2+n > K_1$ we have $0 < \alpha_{c_{n,\epsilon}} < \epsilon$,
		which is condition \eqref{eq-Trapped} of Theorem \ref{TheoremApproximate}.
		Moreover, since $K_1+K_2+n > K_2$, we also have 
		\begin{equation}\label{eq-Finally}
			|\alpha_k - \alpha_{c_{n,\epsilon}-k}|<\epsilon
		\end{equation}
		for $k < c_{n,\epsilon}$, which is condition \eqref{eq-AddingIndices} from Theorem \ref{TheoremApproximate}.
		Finally, since $c_{n,\epsilon} > n$ we see that \eqref{eq-Finally} holds whenever $1\leq k \leq n$.
		By Theorem \ref{TheoremApproximate}, we conclude that $T$ belongs to $\overline{CSO}$.
	\end{proof}

	\begin{Example}\label{ExampleDistinct}
		Consider the weight sequence $\{\alpha_n\}_{n=1}^{\infty}$ whose first few terms are
		\begin{align*}
			\alpha_1 &= 1, & \alpha_9 &= 1 + \tfrac{1}{3^7} + \tfrac{1}{3^9}, \\
			\alpha_2 &= \tfrac{1}{2}, & \alpha_{10} &= \tfrac{1}{2} + \tfrac{1}{3^6} + \tfrac{1}{3^{10}}, \\
			\alpha_3 &= 1 + \tfrac{1}{3^3}, & \alpha_{11} &= 1 + \tfrac{1}{3^3} + \tfrac{1}{3^5} + \tfrac{1}{3^{11}}, \\
			\alpha_4 &= \tfrac{1}{4}, & \alpha_{12} &= \tfrac{1}{4} + \tfrac{1}{3^{12}}, \\
			\alpha_5 &= 1 + \tfrac{1}{3^3} + \tfrac{1}{3^5}, & \alpha_{13} &= 1 + \tfrac{1}{3^3} + \tfrac{1}{3^{13}}, \\
			\alpha_6 &= \tfrac{1}{2} + \tfrac{1}{3^6}, & \alpha_{14} &= \tfrac{1}{2} + \tfrac{1}{3^{14}}, \\
			\alpha_7 &= 1 + \tfrac{1}{3^7}, & \alpha_{15} &= 1 + \tfrac{1}{3^{15}}, \\
			\alpha_8 &= \tfrac{1}{8}, & \alpha_{16} &= \tfrac{1}{16}.
		\end{align*}
		In other words, the weights are defined inductively according to the following rules.
		Let $\alpha_{2^n} = \frac{1}{2^n}$ and, having previously defined $\alpha_1,\alpha_2,\ldots,\alpha_{2^n-1}$, set
		\begin{equation*}
			\alpha_{2^n+j} = \alpha_{2^n - j} + \frac{1}{3^{2^n+j}}. \tag{$1 \leq j \leq 2^n$}
		\end{equation*}
		By construction, the weight sequence $\{\alpha_n\}_{n=1}^{\infty}$ satisfies the hypotheses of Corollary
		\ref{CorollaryKakutani} and hence the corresponding weighted shift $T$ belongs to $\overline{CSO}$.
		Moreover, a simple number-theoretic argument reveals that the $\alpha_i$ are distinct whence
		$\sigma(|T|) = \{0\} \cup \{ \alpha_i\}_{i=1}^{\infty}$
		where each $\alpha_i$ which is not a power of two is an eigenvalue of multiplicity one.  The essential
		spectrum $\sigma_{\text{e}}(|T|)$ of $|T|$ is simply $\{0,1,\frac{1}{2},\frac{1}{4},\ldots\}$.
	\end{Example}
		
	Returning briefly to the subject of compact operators, we remark that 
	the preceding example demonstrates that the fact that the eigenvalues of $|T|$
	tend to zero is essential in the proof of Theorem \ref{TheoremCompact}.  We remark that this fact is used explicitly
	in equation \eqref{eq-DecreasingToZero}.  If the eigenvalues of $|T|$ are allowed to accumulate elsewhere, then
	behavior such as that exhibited in Example \ref{ExampleDistinct} is possible.	
	It is therefore difficult to conceive of a way in which the proof of Theorem 
	\ref{TheoremCompact} could be generalized to include certain classes non-compact operators.

	Having made our remarks about Theorem \ref{TheoremApproximate},
	we now proceed to its proof.

	\begin{proof}[Pf.~of Theorem \ref{TheoremApproximate}]
		Since this proof is somewhat long and intricate, let us first describe the general strategy.
		Using an iterative procedure, we first approximate the original 
		irreducible weighted shift $T$ by a certain direct sum $T'$
		of finite-dimensional matrices of the form \eqref{eq-Palindrome}.
		In general, $T'$ itself will not be a complex symmetric operator since there is
		no reason to believe that the matrices \eqref{eq-Palindrome} produced will have any palindromic
		structure.  We therefore approximate $T'$ with a complex symmetric
		weighted shift $T''$ constructed using an index juggling scheme.
		
		Our first task is to select a strictly increasing sequence $\{ m_k \}_{k=0}^{\infty}$ of indices so that the weighted
		shift $T'$ having the weight sequence $\{ \beta_i \}_{i=1}^{\infty}$ defined by
		\begin{equation}\label{eq-NewWeight}
			\beta_i = 
			\begin{cases}
				\alpha_i &\text{if $i \neq m_k$ for all $k$},\\
				0 & \text{if $i = m_k$ for some $k$},
			\end{cases}
		\end{equation}
		approximates $T$ well in the operator norm while also being itself well-approximated
		by a complex symmetric weighted shift.  
	
		Given $\epsilon > 0$, find an index $N$ such that 
		\begin{equation}\label{eq-Ne2}
			0<\alpha_N < \frac{ \epsilon }{4}.
		\end{equation}
		This is made possible by the assumption \eqref{eq-Trapped}.
		Now inductively define sequences $\{ \delta_k \}_{k=0}^{\infty}$ and $\{ m_k \}_{k=0}^{\infty}$ by setting
		\begin{equation}\label{eq-InitialSettings}
			m_{-1} = 0, \quad m_0 = m_1 = N,
		\end{equation}
		and
		\begin{align}
			m_{2k+3} &= c_{3m_{2k, \delta_k}} - m_{2k-1}, \label{eq-Odd}\\
			m_{2k+2} &= m_{2k+3} - m_{2k} + m_{2k-1}, \label{eq-Even}
		\end{align}
		and
		\begin{equation}\label{eq-dk}
			\delta_k = \frac{1}{8}\min\left\{\alpha_1,\alpha_2,\ldots,\alpha_{3m_{2k}}, \frac{\epsilon}{2^k}\right\}.
		\end{equation}
		Unfortunately, it is not clear that the sequence $\{ m_k \}_{k=1}^{\infty}$ is
		strictly increasing.  We must therefore establish the following claim.
	
		\begin{Claim}
			The sequence $m_1,m_2,m_3,\ldots$ is strictly increasing.
		\end{Claim}
		
		\begin{proof}[Pf.~of Claim]  
			We induct on $k$ in the statement
			\begin{equation}\label{eq-Induction}
				0<m_{2k-1} < m_{2k} < m_{2k+1} < m_{2k+2}.
			\end{equation}
			Let us first verify the base case $k=1$, which is the statement
			\begin{equation}\label{eq-BaseCase}
				m_1 < m_2 < m_3 < m_4.
			\end{equation}		
			First observe that 	
			\begin{equation}\label{eq-cm}
				3m_{2k} \leq c_{3m_{2k},\delta_k}
			\end{equation}
			for $k \geq 0$.  Substituting \eqref{eq-Odd} into \eqref{eq-Even} and using \eqref{eq-cm} then yields
			\begin{equation}\label{eq-DoublingBound}
				2m_{2k} \leq m_{2k+2} 
			\end{equation}
			for $k \geq 0$.	Therefore
			\begin{align*}\qquad\qquad
				m_0&=m_1 = N && \text{by \eqref{eq-InitialSettings}}\\
				&< 2m_0=3m_0-m_0 \\
				&\leq c_{3m_0,\delta_0} - m_0 = m_2 &&\text{by \eqref{eq-cm}} \\
				& = m_3-m_0 < m_3 &&\text{by \eqref{eq-Even}}\\
				&= m_2 + m_0 < 2m_2 && \text{by \eqref{eq-Even}}\\
				&\leq m_4.&& \text{by \eqref{eq-DoublingBound}}.
			\end{align*}
			This establishes the base case \eqref{eq-BaseCase}.
			
			Suppose now that \eqref{eq-Induction} holds for some $k\geq 1$.  Under this hypothesis, we wish to show that
			\begin{equation}\label{eq-DesiredInduction}
				m_{2k+1} < m_{2k+2} < m_{2k+3} < m_{2k+4}.
			\end{equation}
			First note that $m_{2k+1} < m_{2k+2}$ is already part of the induction hypothesis 
			\eqref{eq-Induction}.  The middle inequality of \eqref{eq-DesiredInduction} follows from \eqref{eq-Even} since
			\begin{equation*}
				m_{2k+2} = m_{2k+3} - (m_{2k} - m_{2k-1}) < m_{2k+3}
			\end{equation*}
			holds by the lower inequality in \eqref{eq-Induction}.  To complete the induction, we need only
			verify the upper inequality in \eqref{eq-DesiredInduction}.  This is established as follows:
			\begin{align*}\qquad\qquad
				m_{2k+3}
				&< m_{2k+3} + m_{2k-1} \\
				&= m_{2k+2} + m_{2k} && \text{by \eqref{eq-Even}}\\
				&< 2m_{2k+2} && \text{by \eqref{eq-Induction}}\\
				&\leq m_{2k+4}. && \text{by \eqref{eq-DoublingBound}}
			\end{align*}
			This completes the proof of the claim.
		\end{proof}
		
		Having constructed the desired sequence $\{ m_k \}_{k=0}^{\infty}$ of indices, we consider the weighted
		shift $T'$ whose weight sequence is defined by \eqref{eq-NewWeight}.  To prove that $T'$ is a good
		approximation to $T$ with respect to the operator norm, we must establish that each omitted weight
		$\alpha_{m_k}$ is small.  This is our next task.
		
		In light of \eqref{eq-Ne2} and \eqref{eq-InitialSettings} we have
		\begin{equation}\label{eq-Stg01}
			0< \alpha_{m_0} = \alpha_{m_1} = \alpha_N < \frac{\epsilon}{4}.
		\end{equation}
		Since $m_3 = c_{3N,\delta_0}$ it follows from \eqref{eq-Trapped} and \eqref{eq-dk} that
		\begin{equation}\label{eq-Stg02}
			0< \alpha_{m_3} < \delta_0 < \frac{\epsilon}{4}.
		\end{equation}
		By  \eqref{eq-Odd} and \eqref{eq-Even} we see that
		\begin{equation*}
			m_{2k+2} + m_{2k} = m_{2k+3} + m_{2k-1} = c_{3m_{2k}, \delta_k}
		\end{equation*}
		which yields 
		\begin{align*}
			|\alpha_{m_{2k+3}} - \alpha_{m_{2k-1}}| &< \delta_k < \frac{\epsilon}{2^{k+3}},\\
			|\alpha_{m_{2k+2}} - \alpha_{m_{2k}}| &< \delta_k  < \frac{\epsilon}{2^{k+3}},
		\end{align*}
		by \eqref{eq-AddingIndices}.
		Using the triangle inequality and summing a finite geometric series yields
		\begin{align*} 
			|\alpha_{m_{2k+2}} - \alpha_{m_0}| &< \frac{\epsilon}{4}, &&\text{for $k\geq 0$}, \\
			|\alpha_{m_{2k+3}} - \alpha_{m_1}| &< \frac{\epsilon}{4}, && \text{if $2 \nmid k$},\\
			|\alpha_{m_{2k+3}} - \alpha_{m_3}| &< \frac{\epsilon}{4}, && \text{if $2 \mid k$}.
		\end{align*}
		Since $\alpha_{m_0}, \alpha_{m_1}, \alpha_{m_3} < \frac{\epsilon}{4}$ by \eqref{eq-Stg01}
		and \eqref{eq-Stg02}, we conclude from the preceding that
		\begin{equation*}
			0 < \alpha_{m_k} < \frac{\epsilon}{2}
		\end{equation*}
		for $k \geq 0$.  This implies that $\norm{T - T'} < \frac{\epsilon}{2}$.
		
		Unfortunately, there is no reason to believe that $T'$ belongs to $CSO$.  Therefore our
		next task is to approximate $T'$ with a complex symmetric weighted shift $T''$.
		At this point, it becomes more convenient to write
		$T' = \bigoplus_{k=1}^\infty A_k$
		where
		\begin{equation*}\small
			A_1 = 
			\begin{pmatrix}
				0 &  & & & & \\
				\alpha_{1} & 0 &  & & &\\
				& \alpha_{2} & 0 & &\\
				& &  \ddots & \ddots & & \\
				& & & \alpha_{m_1-1} & 0  & 
			\end{pmatrix}
		\end{equation*}
		and
		\begin{equation*}\small
			A_k = 
			\begin{pmatrix}
				0 &  & & & & \\
				\alpha_{m_{k-1}+1} & 0 &  & & &\\
				& \alpha_{m_{k-1}+2} & 0 & &\\
				& &  \ddots & \ddots & & \\
				& & & \alpha_{m_k-1} & 0  & 
			\end{pmatrix}
		\end{equation*}
		for $k \geq 2$.  To make certain formulas work out, we let $A_0 = A_1$.  Let
		\begin{equation*}\small
			A_k'=
			\begin{pmatrix}
				0 &  & & & & \\
				\alpha_{m_k-1} \ & 0 &  & & &\\
				& \alpha_{m_k-2} & 0 & &\\
				& &  \ddots & \ddots & & \\
				& & &  \alpha_{m_{k-1}+1}& 0  & 
			\end{pmatrix}
		\end{equation*}
		denote the matrix obtained from $A_k$ by reversing the order of the weights along the first subdiagonal.
		
		Consider the relationship between the matrices $A'_{2k+3}$ and $A_{2k}$. 
		The $\ell$th subdiagonal entry of $A'_{2k+3}$ is $\alpha_{m_{2k+3} - \ell}$ while the 
		$\ell$th subdiagonal entry of $A_{2k}$ is $\alpha_{m_{2k-1} + \ell}$.  		
		Since the sum of these indices is
		\begin{equation*}
			(m_{2k+3} - \ell) + (m_{2k-1} + \ell) = m_{2k+3} + m_{2k-1} = c_{3m_{2k},\delta_k}
		\end{equation*}
		by \eqref{eq-Odd}, it follows from \eqref{eq-AddingIndices} that
		\begin{equation*} 
			|\alpha_{m_{2k+3} - \ell} - \alpha_{m_{2k-1} + \ell}| < \delta_k < \frac{\epsilon}{2}.
		\end{equation*}
		In particular, this tells us that
		\begin{equation*}
			\norm{A'_{2k+3} - A_{2k}} < \frac{\epsilon}{2}.
		\end{equation*}
		Since $A_0 = A_1$ we observe that
		\begin{equation*} 
			T' =  \bigoplus_{k=1}^\infty A_k =
			 A_0 \oplus \bigoplus_{k=2}^\infty  A_k 
			 \cong \bigoplus_{k=0}^\infty (A_{2k+3} \oplus A_{2k}) = S'.
		\end{equation*}
		Finally define $T''$ and $S''$ by
		\begin{equation*} 
			T''
			= A'_3 \oplus \left(\bigoplus_{j=2}^\infty \begin{cases} A_j & \text{if $2 \nmid j$}, \\ A'_{j+3} & \text{if $2 \mid j$}, \end{cases}\right)
			\cong \bigoplus_{k=0}^\infty \underbrace{(A_{2k+3} \oplus A'_{2k+3})}_{\in CSO}  
			= S''.
		\end{equation*}
		The operator $S''$ belongs to $CSO$ since it is a direct sum of matrices $A_{2k+3} \oplus A'_{2k+3}$ of the form
		\eqref{eq-Palindrome} whose entries on the first subdiagonal are palindromic.
		Therefore
		\begin{equation*}
			\norm{T' - T''} = \norm{S' - S''} < \frac{\epsilon}{2}
		\end{equation*}
		whence
		\begin{equation*}
			\norm{T - T''} \leq \norm{T- T'} + \norm{T' - T''} < \frac{\epsilon}{2} + \frac{\epsilon}{2} = \epsilon.
		\end{equation*}
		Thus $T$ belongs to $\overline{CSO}$, as claimed.
	\end{proof}

\section*{Acknowledgments}

We wish to thank W.R.~Wogen for his numerous comments and suggestions.
We also greatly appreciate the careful eye of the anonymous referee, who pointed out
a number of minor mistakes in the original manuscript.

\bibliography{OCCSO}

\end{document}